\documentclass{amsart}
\usepackage[english]{babel}
\usepackage[cp1250]{inputenc}
\usepackage{amssymb,amsthm,amsfonts}
\usepackage{verbatim}
\usepackage{url}

\usepackage{tikz-cd}
\usepackage{amsmath,amssymb,amsfonts, amscd}
\usepackage{pictexwd,dcpic}
\usepackage{amsthm}
\usepackage{color}
\usepackage[T1]{fontenc}
\usepackage{verbatim}
\usepackage{epstopdf}
\usepackage{enumerate}
\usepackage{url}
\usepackage[toc,page]{appendix}

\usepackage{xcolor}

\theoremstyle{theorem}
\newtheorem{theorem}{Theorem}[section]
\newtheorem{lemma}[theorem]{Lemma}
\newtheorem{corollary}[theorem]{Corollary}
\newtheorem{proposition}[theorem]{Proposition}
\theoremstyle{definition}
\newtheorem{definition}{Definition}
\newtheorem{remark}{Remark}
\newtheorem{example}{Example}
\newtheorem{construction}{Construction}

\newcommand{\spa}{\operatorname{span}}

\renewcommand{\AA}{\mathbb A}
\newcommand{\LL}{\mathbb L}
\newcommand{\RR}{\mathbb R}
\newcommand{\NN}{\mathbb N}
\newcommand{\ZZ}{\mathbb Z}

\newcommand{\PP}{\mathbb P}

\newcommand{\spec}{\mathrm{Spec}}

\newcommand{\lan}{\langle}
\newcommand{\ran}{\rangle}

\newcommand{\lam}{\lambda}

\newcommand{\h}{\operatorname{Hom}}

\newcommand{\tot}{\operatorname{tot}}
\newcommand{\im}{\operatorname{im}}
\renewcommand{\sp}{\operatorname{Span}}

\newcommand{\HH}{\operatorname{HH}}
\newcommand{\un}{\underline}
\newcommand{\sgn}{\operatorname{sgn}}

\newcommand{\ffi}{\varphi}

\newcommand{\kss}{\scriptscriptstyle}
\newcommand{\kbb}{{\kss \bullet}}

\begin{document}

\title[The Gerstenhaber product of affine toric varieties]{The Gerstenhaber product $\HH^2(A)\times \HH^2(A)\to \HH^3(A)$ of affine toric varieties} 
\author{Matej Filip}
\address{JGU Mainz, Institut fur Mathematik, Staudingerweg 9, 55099 Mainz, Germany}
\email{mfilip@uni-mainz.de}

\begin{abstract}
For an affine toric variety $\spec(A)$, we give a convex geometric interpretation of the Gerstenhaber product $\HH^2(A)\times \HH^2(A)\to \HH^3(A)$ between the Hochschild cohomology groups. In the case of Gorenstein toric surfaces we prove that the Gerstenhaber product is the zero map.
As an application in commutative deformation theory we find the equations of the versal base space (in special lattice degrees) up to second order  for not necessarily isolated toric Gorenstein singularities. Our construction reproves and generalizes results obtained in \cite{alt2} and \cite{sle}.

\end{abstract}


\maketitle

\section{Introduction}

It is well known that non-commutative deformations of an affine variety $X=\spec(A)$ are controlled by the Hochschild differential graded Lie algebra (dgla for short). There are two important $A$-modules:  
the second Hochschild cohomology group $\HH^2(A)$ that describes the first order deformations and the third Hochschild cohomology group $\HH^3(A)$ that contains the obstructions for extending deformations of $X$ to larger base spaces.

$\HH^2(A)$ can be decomposed as $H^2_{(1)}(A)\oplus H^2_{(2)}(A)$, where $H^2_{(1)}(A)$ describes the first order commutative deformations. Moreover, there exists the Harrison dgla, that is a sub-dgla of the Hochschild dgla, controlling the commutative deformations of $X$.

Focusing on commutative deformations, computing the versal deformation of affine varieties with isolated singularities is a challenging problem. 
For toric surfaces
Koll\'ar and Shepherd-Barron \cite{kol-she} showed that there is a correspondence between certain partial resolutions (P-resolutions) and the reduced versal base space components. 
Furthermore, in \cite{chr} and \cite{ste2} Christophersen and Stevens gave a set of equations for each reduced component of the versal base space.  
For higher dimensional toric varieties the versal base space was computed by Altmann \cite{alt3} in the case 
of isolated toric Gorenstein singularities.

In order to better understand the deformation theory of $X=\spec(A)$, we need to understand the cup product $T^1(A)\times T^1(A)\to T^2(A)$ between Andre-Quillen cohomology groups. The associated quadratic form describes the equations of the versal base space (if exists) up to second order.  A formula for computing the cup product for toric varieties that are smooth in codimension 2 was obtained in \cite{alt2}. Since this formula is especially simple in the case of three-dimensional isolated toric Gorenstein singularities, it helped Altmann to construct the versal base space in \cite{alt3}. 
The cup product of toric varieties was also analysed by Sletsj\o e \cite{sle} but unfortunately there is a mistake in the paper (see Section 3). 

The cup product is coming from the differential graded Lie algebra (dgla for short) arising from the cotangent complex. This dgla is isomorphic to the Harrison dgla. The Lie bracket induces the product $H^2_{(1)}(A)\times H^2_{(1)}(A)\to H^3_{(1)}(A)$ between the Harrison cohomology groups, which is isomorphic to the cup product $T^1(A)\times T^1(A)\to T^2(A)$. 

In this paper we give a convex geometric description of the Harrison product for an affine toric variety $\spec(A)$. This gives us a general cup product formula $T^1(A)\times T^1(A)\to T^2(A)$ (without the assumption of smoothness in codimension 2) that agrees in the case of Gorenstein isolated singularities with Altmann's cup product formula. We obtain a nice expression of the cup product especially for Gorenstein not necessarily isolated singularities. This gives us an idea how the versal base space in special lattice degrees could look like. Note that since we are dealing with the non-isolated case, the $T^1(A)$ is non-zero in infinitely many lattice degrees. 

We also generalize the above description to the product $\HH^2(A)\times \HH^2(A)\to \HH^3(A)$, induced by the Lie bracket (also called the Gerstenhaber bracket) of the Hochschild dgla. This product  is also known as the Gerstenhaber product. 
As an application we obtain that this product is zero for Gorenstein toric surfaces. This is interesting since it might lead to the formality theorem for (singular) Gorenstein toric surfaces. The formality theorem has been proved for smooth affine varieties (see \cite{kon2}, \cite{dol-tam-tsy}).

The paper is organized as follows: in Section 2 we recall deformation theory of toric varieties. In Section 3 we give a convex geometric description of the product $H^2_{(1)}(A)\times H^2_{(1)}(A)\to H^3_{(1)}(A)$ for toric varieties. The cup product in the special case of toric Gorenstein  singularities is computed in Section 4 (see Theorem \ref{prop cup iso} and Subsection 4.2), where we also show that our product agrees with Altmann's cup product formula for isolated toric Gorenstein singularities (see Corollary \ref{cor altmann}). We describe the quadratic equations of the versal base space in the Gorenstein degree $-R^*$ in Corollary \ref{cor quad}. In Section 5 we analyse the Gerstenhaber product $\HH^2(A)\times \HH^2(A)\to \HH^3(A)$ for toric varieties. The proof that this product is the zero map for Gorenstein toric surfaces is done in Proposition \ref{pr gor sur}.

\section{Preliminaries}

\subsection{Toric geometry}
Let $k$ be a field of characteristic 0. Let $M,N$ be mutually dual, finitely generated, free Abelian groups. We denote by $M_{\RR}$, $N_{\RR}$ the associated real vector spaces obtained via base change with $\RR$. Let $\sigma=\lan a_1,...,a_N \ran\subset N_{\RR}$ be a rational, polyhedral cone with apex in $0$ and let $a_1,...,a_N\in N$ denote its primitive fundamental generators (i.e. none of the $a_i$ is a proper multiple of an element of $N$). We define the dual cone $\sigma^{\vee}:=\{r\in M_{\RR}~|~\lan \sigma,r \ran\geq 0\}\subset M_{\RR}$ and denote by $\Lambda:=\sigma^{\vee}\cap M$ the resulting semi-group of lattice points. Its spectrum $\text{Spec}(k[\Lambda])$ is called an \emph{affine toric variety}.

\subsection{The Hochschild dgla}
For any finitely generated $k$-algebra, we can define the cotangent complex $\LL_{A|k}$ and its 
 derived exterior powers $\wedge^i\LL_{A|k}$ (see e.g.\ \cite{lod}).
 The $n$-th cohomology group of $\h_A(\wedge^i\LL_{A|k},A)$ is called the $n$-th (higher) \emph{Andr\'e-Quillen cohomology} group and denoted by $T^n_{(i)}(A)$. We will also use the following notation $T^n(A):=T^n_{(1)}(A)$ for $n\geq 1$.

Using the notation from \cite{fil} we denote by $C^{\kbb}(A)$ the Hochschild cochain complex and by $C^n(A)=C^n_{(1)}(A)\oplus \cdots \oplus C^n_{(n)}(A)$ the Hodge decomposition which induces the decomposition in cohomology
$\HH^n(A)\cong H^n_{(1)}(A)\oplus \cdots \oplus H^n_{(n)}(A),$
where $\HH^n(A)$ is the $n$-th Hochschild cohomology group and $H^n_{(i)}(A)$  is the $n$-th cohomology of $C^{\kbb}_{(i)}(A)$.
It is well known  (see \cite[Proposition 4.5.13]{lod}) that $T^{n-i}_{(i)}(A)\cong H^n_{(i)}(A)$ for all $i\geq 1$.

In order to get a dgla structure on the Hochschild cochain complex we need to shift it by $1$. 
The Lie bracket $[\cdot,\cdot]:C^n(A)\times C^m(A)\to C^{m+n-1}(A)$, which is also called the \emph{Gerstenhaber bracket}, is well known so we skip the definition of it (see e.g.\ \cite[Section 2]{fil}). 
In particular, the Gerstenhaber bracket induces the product
\begin{equation}\label{eq ger pr pr}
[\cdot,\cdot]:\HH^2(A)\times \HH^2(A)\to \HH^3(A),
\end{equation}
between the important $A$-modules mentioned in Introduction.
The product in \eqref{eq ger pr pr} is called the \emph{Gerstenhaber product}.

We denote the projectors of $\HH^n(A)$ to $H^n_{(i)}(A)$ by $e_n(i)$.
\begin{lemma}\label{lem pal lem}
For an element $p\in H^2_{(2)}(A)$ and an element $q\in H^2_{(1)}(A)$ we have the following: 
\begin{itemize}
\item the equation $e_3(3)[p,p]=0$ is the Jacobi identity, $e_3(2)[p,p]=0$
\item $[p,q]=e_3(2)[p,q]$ and $[q,q]=e_3(1)[q,q]$.
\end{itemize}
\end{lemma}
\begin{proof}
An easy computation, see also \cite{pal2}.
\end{proof}

Using Lemma \ref{lem pal lem}, we see that the Gerstenhaber product consists of the products $H^2_{(i)}(A)\times H^2_{(1)}(A)\to H^3_{(i)}(A)$, for $i=1,2$ and $H^2_{(2)}(A)\times H^2_{(2)}(A)\to \HH^3(A)$. 

In \cite{fil} we showed that every Poisson structure $p\in H^2_{(2)}(A)$ on an affine toric variety $X_{\sigma}=\spec(A)$ can be quantized, which implies that $[p,p]=0\in \HH^3(A)$. Note also that for an element $p\in H^2_{(2)}(A)$ that is not a Poisson structure, Lemma \ref{lem pal lem} implies that $[p,p]\ne 0$.
In this paper we will focus in understanding the remaining two products $H^2_{(i)}(A)\times H^2_{(1)}(A)\to H^3_{(i)}(A)$, for $i=1,2$ and $A$ an affine toric variety.

\subsection{The Hochschild dgla of toric varieties}

From \cite{fil} we recall the following. 
In the group ring of the permutation group $S_n$ one defines the shuffle $s_{i,n-i}$ to be $\sum (\sgn \pi)\pi$, where the sum is taken over those permutations $\pi\in S_n$ such that $\pi (1) <\pi (2)<\cdots<\pi (i)$ and $\pi(i+1)<\pi(i+2)<\cdots <\pi(n)$. 
Let $s_n=\sum_{i=1}^{n-1}s_{i,n-i}$.
\begin{definition}\label{def monf}
$L\subset \Lambda$ is said to be \emph{monoid-like} if for all elements $\lam_1,\lam_2\in L$ the relation $\lam_1-\lam_2\in \Lambda$ implies $\lam_1-\lam_2\in L$. Moreover, a subset $L_0\subset L$ of a monoid-like  set is called \emph{full} if $(L_0+\Lambda)\cap L=L_0.$ 
\end{definition}
For any subset $P\subset \Lambda$ and $n\geq 1$ we introduce $S_n(P):=\{(\lam_1,...,\lam_n)\in P^n~|~\sum_{v=1}^n\lam_v\in P\}$. If $L_0\subset L$ are as in the previous definition, then this gives rise to the following vector spaces ($1\leq i\leq n$):
$$C^n_{(i)}(L,L\setminus L_0;k):=\{\varphi: S_n(L)\to k~|~\ffi\circ s_n=(2^i-2)\varphi, \,\varphi \,\text{ vanishes on }\,S_n(L\setminus L_0)\},$$
which turn into a complex with the differentials 
\begin{equation}\label{def dif d}
d:C_{(i)}^{n-1}(L,L\setminus L_0;k)\to C_{(i)}^n(L,L\setminus L_0;k),
\end{equation}
$$(d\varphi)(\lam_1,...,\lam_n):=$$ $$\varphi(\lam_2,...,\lam_n)+\sum^{n-1}_{i=1}(-1)^{i}\varphi(\lam_1,...,\lam_i+\lam_{i+1},...,\lam_n)+(-1)^n\varphi(\lam_1,...,\lam_{n-1}).$$

\begin{definition}
By $H^n_{(i)}(L,L\setminus L_0;k)$ we denote the cohomology groups of the above complex $C_{(i)}^{\kbb}(L,L\setminus L_0;k)$. We denote $H_{(i)}^n(L,\emptyset;k)$ shortly by $H_{(i)}^n(L;k)$.
\end{definition}

It is a trivial check that for $A=k[\Lambda]$ the Hochschild differentials respect the grading given by the degrees $R\in M$. Thus we get the Hochschild subcomplex $C^{\kbb,R}_{(i)}$ and we denote the corresponding cohomology groups by $H^{n,R}_{(i)}(A)\cong T^{n-i,R}_{(i)}(A)$. When an algebra $A$ will be clear from the context, we will also write $H^n_{(i)}(R)$. It holds that $\HH^n(A)=\bigoplus_{R\in M}\HH^{n,R}(A)$ and $\HH^{n,R}(A)\cong \oplus_iH^{n,R}_{(i)}(A)$. Thus we can analyse the Hochschild cohomology groups by analysing them in every degree $R\in M$.

For an element $R\in M$ we denote $\Lambda(R):=\Lambda+R$.
\begin{proposition}\label{prop 1}
Let $A=k[\Lambda]$ and $R\in M$. It holds that
\begin{equation}\label{sisom}
H^{n,-R}_{(i)}(A)\cong T_{(i)}^{n-i,-R}(A)\cong H_{(i)}^{n}(\Lambda,\Lambda\setminus \Lambda(R);k).
\end{equation}
\end{proposition}  
\begin{proof}
See \cite[Proposition 4.2]{fil}.
\end{proof}
Proposition \ref{prop 1} tells us how to compute the Hochschild comology groups in degree $-R$. 

The next lemma describes the Gerstenhaber product in the toric setting.
\begin{lemma}\label{lem ger pr}
For each $i\in\{1,2\}$ the Gerstenhaber bracket induces the product 
$$
[\cdot,\cdot]: C^2_{(i)}(\Lambda,\Lambda\setminus \Lambda(R);k)\times C^2_{(1)}(\Lambda,\Lambda\setminus \Lambda(S);k)\to C^3_{(i)}(\Lambda,\Lambda\setminus \Lambda(R+S);k),
$$
$$[f,g]= 
\begin{array}{c}
f(-S+\lam_1+\lam_2,\lam_3)g(\lam_1,\lam_2)-f(\lam_1,-S+\lam_2+\lam_3)g(\lam_2,\lam_3)+\\
g(-R+\lam_1+\lam_2,\lam_3)f(\lam_1,\lam_2)-g(\lam_1,-R+\lam_2+\lam_3)f(\lam_2,\lam_3).
\end{array}
$$
This product induces the Gerstenhaber product in cohomology 
\begin{equation}\label{eq cup tp}
H^2_{(i)}(\Lambda,\Lambda\setminus \Lambda(R);k)\times H^2_{(1)}(\Lambda,\Lambda\setminus \Lambda(S);k)\to H^3_{(i)}(\Lambda,\Lambda\setminus \Lambda(R+S);k).
\end{equation}
\end{lemma}
\begin{proof}
See \cite[Lemma 5.4]{fil}.
\end{proof}

\begin{remark}
Note that in \cite[Lemma 5.4]{fil} we obtained a complete description of the Gerstenhaber product (also for other parts of the Hodge decomposition). In general the Gerstenhaber product does not respect the Hodge decomposition like in Lemma \ref{lem ger pr}. For us only the two products in Lemma \ref{lem ger pr} will be important. In \eqref{eq cup tp} we describe the Gerstenhaber product 
$H^{2,-R}_{(i)}(k[\Lambda])\times H^{2,-S}_{(1)}(k[\Lambda])\to H^{3,-R-S}_{(i)}(k[\Lambda])$.
The Gerstenhaber product always respects the toric grading (i.e. two elements of degree $R$ and $S$ are mapped to an element of degree $R+S$). Thus we can analyse the Gerstenhaber product by analysing it in toric degrees.
\end{remark}

In the following we will recall from \cite{fil} how the groups appearing in \eqref{eq cup tp} can be nicely interpreted. This will later lead to  a nice interpretation of the Gerstenhaber product.

For a face $\tau$ in  $\sigma$ (denoted $\tau \leq \sigma$) we define the convex sets introduced in \cite{klaus}:
\begin{equation}\label{eq con set} 
K_{\tau}^R:=\Lambda\cap (R-\mathrm{int}\tau^{\vee}). 
\end{equation}
The above convex sets admit the following properties ($\sigma=\lan a_1,...,a_N\ran$): 
\begin{itemize}
\item $K_0^R=\Lambda$ and $K_j^R:=K_{a_j}^R=\{r\in \Lambda~|~\lan a_j,r \ran<\lan a_j,R \ran\}$ for $j=1,...,N$.
\item For $\tau\ne 0$ the equality $K^R_{\tau}=\cap_{a_j\in \tau}K_{a_j}^R$ holds.
\item $\Lambda \setminus (R+\Lambda)=\cup_{j=1}^NK_{a_j}^R$.
\end{itemize}

\begin{example}\label{ex tri intro}
Let $a_1=(-1,2)$ and $a_2=(1,2)$.  
Let $\sigma=\lan a_1,a_2\ran$ and thus the Hilbert basis of $\Lambda=\sigma^\vee\cap M$ is 
$$H=\{(-2,1),(-1,1),(0,1),(1,1),(2,1)\}.$$ 
Let $R=(0,1)$. 
We have $\Lambda\setminus \Lambda(R)=K^R_{a_1}\cup K^R_{a_2}$, where 
$$K^R_{a_1}=\{(r_1,r_2)\in \Lambda~|~-r_1+2r_2<2\}=\{(0,0),(2,1),(1,1),(4,2),(3,2),...\},$$
$$K^R_{a_2}=\{(r_1,r_2)\in \Lambda~|~r_1+2r_2<2\}=\{(0,0),(-2,1),(-1,1),(-4,2),(-3,2),...\}.$$
\end{example}
For each $i\geq 1$ we have the following important double complexes (see \cite[Section 4.2]{fil}). We define 
$C^q_{(i)}(K_{\tau}^R;k):=C^q_{(i)}(K_{\tau}^R,\emptyset;k)$ and
$$C^q_{(i)}(K_p^R;k):=\oplus_{\tau\leq \sigma,\dim \tau=p}C^q_{(i)}(K^R_{\tau};k)~~~~(0\leq p\leq \dim \sigma,~q\geq i).$$
The differentials 
\begin{equation}\label{eq delta}
\delta: C^q_{(i)}(K^R_{p};k)\to C_{(i)}^q(K_{p+1}^R;k)
\end{equation}
 are defined in the following way: we are summing (up to a sign) the images of the restriction map $C^q_{(i)}(K_{\tau}^R;k)\to C^q_{(i)}(K_{\tau'}^R;k)$, for any pair $\tau\leq \tau'$ of $p$ and $(p+1)$-dimensional faces, respectively. The sign arises from the comparison of the (pre-fixed) orientations of $\tau$ and $\tau'$ (see also \cite[pg. 580]{cox} for more details).
For each $i\geq 1$ this construction gives us the double complexes that we shortly denote by 
\begin{equation}\label{eq im double}
C^{\kbb}_{(i)}(K^R_{\kbb};k).
\end{equation}

\begin{proposition}\label{th 23}
$T^{n-i,-R}_{(i)}(A)=H^{n}\big (\tot^{\kbb}(C_{(i)}^{\kbb}(K^R_{\kbb};k))\big )$ for $1\leq i\leq n$.
\end{proposition}
\begin{proof}
See \cite[Proposition 4.4]{fil}.
\end{proof}

\begin{proposition}\label{smooth}
If $\tau\leq \sigma$ is a smooth face, then $H^q_{(i)}(K^R_{\tau};k)=0$ for $q\geq i+1$.
\end{proposition}
\begin{proof}
See \cite[Proposition 4.6]{fil}.
\end{proof}

\begin{theorem}\label{th 24}
The $k$-th cohomology group of the complex
$$0\to H^i_{(i)}(\Lambda;k)\to \bigoplus_jH^i_{(i)}(K^R_j;k)\to \bigoplus_{\tau\leq \sigma, \dim\tau=2}H^i_{(i)}(K^R_\tau;k)\to \bigoplus_{\tau\leq \sigma,\dim\tau=3}H^i_{(i)}(K^R_\tau;k)$$
is isomorphic to $T^{k,-R}_{(i)}(A)$, for $k=0,1,2$ ($H^i_{(i)}(\Lambda;k)$ is the degree $0$ term).
\end{theorem}
\begin{proof}
Follows from the proof of \cite[Theorem 4.6]{fil}.
\end{proof}


\section{The product $H^2_{(1)}(k[\Lambda])\times H^2_{(1)}(k[\Lambda])\to H^3_{(1)}(k[\Lambda])$}

In this section we give a general formula for the product $H^2_{(1)}(k[\Lambda])\times H^2_{(1)}(k[\Lambda])\to H^3_{(1)}(k[\Lambda])$, extending Altmann's cup product formula on toric varieties that are smooth in codimension 2. Note that Altmann obtained the cup product formula with different methods (using Laudal's description, coming from the cotangent complex).
For $R,S\in M$ and $A=k[\Lambda]$ we give a convex geometric description of the product $H^{2,-R}_{(1)}(A)\times H^{2,-S}_{(1)}(A)\to H^{3,-R-S}_{(1)}(A)$, which we call the \emph{Harrison cup product}.

As mentioned in Introduction, the Harrison cup product was also analysed by Sletsj\o e \cite{sle} but unfortunately with a mistake that we point out now.
We start by recalling basic constructions from \cite{sle}.

For $R\in M$ recall that $\Lambda(R):=\Lambda+R$.
 We have an exact sequence of complexes:
$$0\to C^{\kbb}_{(1)}(\Lambda,\Lambda\setminus \Lambda(R);k)\to C^{\kbb}_{(1)}(\Lambda;k)\to C^{\kbb}_{(1)}(\Lambda\setminus \Lambda(R);k)\to 0.$$

Note that $H^q_{(1)}(\Lambda;k)=0$ for $q\geq 2$ by Proposition \ref{smooth}. Thus we can write the corresponding long exact sequence in cohomology and we get the following.
\begin{corollary}\label{prop d}
The sequence
$$0\to H^{1}_{(1)}(\Lambda,\Lambda\setminus \Lambda(R);k)\to H^{1}_{(1)}(\Lambda;k)\to H^{1}_{(1)}(\Lambda\setminus \Lambda(R);k)\xrightarrow{d} H^{2}_{(1)}(\Lambda,\Lambda\setminus \Lambda(R);k)\to 0$$
is exact 
and
$$H^{n}_{(1)}(\Lambda\setminus \Lambda(R);k)\cong H^{n+1}_{(1)}(\Lambda,\Lambda\setminus \Lambda(R);k),$$
for $n\geq 2$. This isomorphism is induced by the map $d$.
\end{corollary}

\begin{remark}
Here with the map $d$ we mean that we first extend a function from $\Lambda\setminus \Lambda(R)$ to the whole of $\Lambda$ by $0$ and then we apply our differential $d$. Both maps we will denote by $d$ and the meaning will be clear from the context.
\end{remark}

\begin{remark}
Let $L_0\subset L$ be as in Definition \ref{def monf}. Elements in $C^1_{(1)}(L;k)=\{f:L\to k\}$ are functions on $L$. If we additionally have $f\in C^1_{(1)}(L,L\setminus L_0;k),$ then a function $f$ vanishes on  $L\setminus L_0$. Restricting a function $f\in C^1_{(1)}(L;k)$ to some monoid-like subset $K\subset L$ means that we look on $f$ as an element in $C^1_{(1)}(K;k)$. Immediately from definition we obtain that $H^1_{(1)}(L;k)$ is the space of functions $f\in C^1_{(1)}(L;k)$ such that $f(a+b)=f(a)+f(b)$ if $a+b\in L$. Thus we call elements from $H^1_{(1)}(L;k)$ \emph{additive functions on $L$}.
\end{remark}

Let $\xi$ be an element from $H_{(1)}^1(\Lambda\setminus \Lambda(R);k)$. We extend (not additively) $\xi$ to the whole of $\Lambda$ by $0$ (i.e.\ $\xi(\lam)=0$ for $\lam\in \Lambda(R))$. This extended function we denote by $\xi^0\in C^1_{(1)}(\Lambda;k)$.
We have $T^{1,-R}(A)\cong H_{(1)}^2(\Lambda,\Lambda\setminus \Lambda(R);k)$ by Proposition \ref{prop 1} and the surjective map $$H_{(1)}^{1}(\Lambda\setminus \Lambda(R);k)\xrightarrow{d} H_{(1)}^{2}(\Lambda,\Lambda\setminus \Lambda(R);k)$$ by Corollary \ref{prop d}. Thus we see that every element of $T^{1}(-R)\cong H^2_{(1)}(-R)$ can be written as $\overline{d\xi^0}$ for some $\xi\in H_{(1)}^1(\Lambda\setminus \Lambda(R);k)$.
Here $\overline{d\xi^0}$ denotes the cohomology class of $d\xi^0\in C^2_{(1)}(\Lambda,\Lambda\setminus \Lambda(R);k)$.

\begin{example}\label{ex tri}
Continuing with Example \ref{ex tri intro},
let $\xi\in H^1_{(1)}(\Lambda\setminus \Lambda(R);k)$, with 
$$\xi(r_1,r_2)=\left\{
\begin{array}{ll}
r_2&\text{ if }(r_1,r_2)\in K^{R}_{a_2}\\
0&\text{ if }(r_1,r_2)\in K^{R}_{a_1}.
\end{array}
\right.
$$
Since $K^{R}_{a_1}\cap K^{R}_{a_2}=\emptyset$, $\xi$ is well defined.
 Note that $R\not\in \Lambda\setminus \Lambda(R)$. For $\xi^0\in C^1_{(1)}(\Lambda;k)$ we have $\xi^0(R)=0$.
\end{example}

From Lemma \ref{lem ger pr} recall the Gerstenhaber product for $i=1$, called the (Harrison) cup product. 

\begin{construction}\label{cons 1}
Let $R,S\in M$ and let $\xi$ and $\mu$ be elements from $H^1_{(1)}(\Lambda\setminus \Lambda(R);k)$ and $H^1_{(1)}(\Lambda\setminus \Lambda(S);k)$, respectively. 
The Harrison cup product $$\overline{[d\xi^0,d\mu^0]}\in H^3_{(1)}(\Lambda,\Lambda \setminus \Lambda(R+S);k)\cong T^{2,-R-S}(A)$$ can be seen as the cohomology class of $[d\xi^0,d\mu^0]\in C_{(1)}^3(\Lambda;k)$ in $\tot^\kbb[C^\kbb_{(1)}(K^{R+S}_{\kbb};k)]$  by Proposition \ref{th 23} (recall that $K^{R+S}_0=\Lambda$).

We can find an element $C\in C^2_{(1)}(\Lambda;k)$ such that $dC=[d\xi^0,d\mu^0]\in C^3_{(1)}(\Lambda;k)$ since $H^3_{(1)}(\Lambda;k)=0$ by Proposition \ref{smooth}. 

We inject $C$ into $(C_1,...,C_N)\in \oplus_{j}C_{(1)}^2(K^{R+S}_{a_j};k)$. There exist functions $G_j\in C^1_{(1)}(K^{R+S}_{a_j};k)$ for $j=1,...,N$, such that $dG_j=C_j$. Indeed, $H^2_{(1)}(K^{R+S}_{a_j};k)=0$ for all $j=1,...,N$ by Proposition \ref{smooth}. 

Let us denote $G:=(G_1,...,G_N)\in \oplus_jC^1_{(1)}(K^{R+S}_{a_j};k)$. Recall the definition of the map $\delta$ from \eqref{eq delta}. By the above construction  $\overline{\delta G}=\overline{[d\xi^0,d\mu^0]}\in \tot^\kbb[C^\kbb_{(1)}(K^{R+S}_{\kbb};k)]$. 
\end{construction}

In the following we try to find the function $G$ explicitly. Then $$\delta G\in \bigoplus_{\tau\leq \sigma, \dim\tau=2}H^1_{(1)}(K^{R+S}_\tau;k).$$ The latter space is easier to work with, which will give us a nice description of the cup product and thus bring us closer to understanding the versal base space.

\begin{proposition}\label{prop sle}
We define $C\in C^2_{(1)}(\Lambda;k)$ as
$$C(\lam_1,\lam_2):=$$
$$\xi^0(\lam_1)\mu^0(\lam_2)+\xi^0(\lam_2)\mu^0(\lam_1)-d\xi^0(\lam_1,\lam_2)\mu^0(-R+\lam_1+\lam_2)-d\mu^0(\lam_1,\lam_2)\xi^0(-S+\lam_1+\lam_2),$$
where $\mu^0(-R+\lam_1+\lam_2):=0$ (resp. $\xi^0(-S+\lam_1+\lam_2):=0$) if $-R+\lam_1+\lam_2\not\in \Lambda$ (resp. $-S+\lam_1+\lam_2\not\in \Lambda$).
It holds that 
$[d\xi^0,d\mu^0]=dC.$

\end{proposition}  
\begin{proof}
See \cite[Theorem 4.8]{sle}.
\end{proof}

Sletsj\o e \cite{sle} claimed that Proposition \ref{prop sle} gives us a nice cup product formula, but unfortunately there is a mistake in his paper: in \cite{sle} it was written that only the first two terms of $C(\lam_1,\lam_2)$ matter for the computations of the cup product formula and that the other two vanish with $d$. 
This is not correct since
$d\xi^0\not\in C^2_{(1)}(\Lambda,\Lambda\setminus\Lambda(R+S);k)$, which was wrongly assumed in the paper. We only have $d\xi^0\in C^2_{(1)}(\Lambda,\Lambda\setminus\Lambda(R);k)$.

Thus we need to consider $C(\lam_1,\lam_2)$ with all 4 terms and we will try to simplify this using the  double complex $C_{(1)}^{\kbb}(K^R_{\kbb};k)$ (see the equation \eqref{eq im double}).

Following Construction \ref{cons 1} we now inject $C$ into $(C_1,...,C_N)\in \oplus_{j}C_{(1)}^2(K^{R+S}_{a_j};k)$. In the following  we will define the functions $h_j$ which will serve as a first approximation of the functions $G_j$ from Construction \ref{cons 1}, i.e. $dh_j(\lam_1,\lam_2)=C_j(\lam_1,\lam_2)$ "almost" holds (we will be more precise later).

For each $j=1,...,N$ we choose $\tilde{\xi}_j\in (M\otimes_\ZZ k)^*$ such that $\tilde{\xi}_j$ restricted to $K^R_{a_j}$ equals $\xi$, i.e. $\tilde{\xi}_j=\xi$ as elements in $H^1_{(1)}(K^R_{a_j};k)$. Note that this is always possible since $H^1_{(1)}(K^R_{a_j};k)$ is isomorphic to $(\spa_k(K^R_{a_j}))^*$, which is the space of $k$-linear functions on $\spa_k(K^R_{a_j})$ (see \cite[Proposition 4.2]{klaus}; for the description of $\spa_k(K^R_{a_j})$ see the discussion after Theorem \ref{th end}). 
Note also that $K^R_{a_j}\subset \Lambda\setminus \Lambda(R)$ so it makes sense to consider $\xi$ restricted to $K^R_{a_j}$.
If  $\lan a_j,R\ran=0$ holds, then a choice of $\tilde{\xi}_j$ is not unique. In the same way we define $\tilde{\mu}_j$.

We define $\xi_j$, $\mu_j\in C^1_{(1)}(M;k)$ as follows:
$$
\xi_j(\lam):=\left\{
\begin{array}{ll}
\tilde{\xi}_j(\lam)&\text{ if }\lam\in K_{a_j}^{R+S}\\
0&\text{ otherwise }
\end{array}
\right.
~~~
\mu_j(\lam):=\left\{
\begin{array}{ll}
\tilde{\mu}_j(\lam)&\text{ if }\lam\in K_{a_j}^{R+S}\\
0&\text{ otherwise. }
\end{array}
\right.
$$
Note that by construction $\xi_j$ and $\mu_j$ are additive functions on $K_{a_j}^{R+S}$ (i.e. for $\lam_1,\lam_2\in K_{a_j}^{R+S}$ we have $\xi_j(\lam_1)+\xi_j(\lam_2)=\xi_j(\lam_1+\lam_2)$ if $\lam_1+\lam_2\in K_{a_j}^{R+S}$ and similarly for $\mu_j$).
Moreover, for each $j=1,...,N$ we define $\xi^0_j, \mu_j^0\in C^1_{(1)}(M;k)$ as

$$
\xi_j^0(\lam):=\left\{
\begin{array}{ll}
\xi_j(\lam)&\text{ if }\lam\in K^{R+S}_{a_j}\cap\big(\Lambda \setminus \Lambda(R))\\
0&\text{ otherwise, }
\end{array}
\right.
$$
$$
\mu_j^0(\lam):=\left\{
\begin{array}{ll}
\mu_j(\lam)&\text{ if }\lam\in K^{R+S}_{a_j}\cap\big(\Lambda \setminus \Lambda(S))\\
0&\text{ otherwise. }
\end{array}
\right.
$$
Note that for $\lam\in \Lambda$ we have
$$
\xi_j^0(\lam)=\left\{
\begin{array}{ll}
\xi_j(\lam)&\text{ if }\lam\in \big(K^{R+S}_{a_j}\setminus K_{a_j}^R\big)\cap\big(\cup_{k;k\ne j}K^{R}_{a_k}\big)\\
\xi^0(\lam)&\text{ otherwise }
\end{array}
\right.
$$
and similarly for $\mu_j^0$. This description explains the notation. It will be used in Example \ref{ex dod3} and Remark \ref{rem th}.

\begin{example}\label{ex dod3}
Continuing with Example \ref{ex tri} we see that $\tilde{\xi}_1$ and $\tilde{\xi}_2$ are unique in this case and we have $\tilde{\xi}_2(\lam_1,\lam_2)=\lam_2$ and $\tilde{\xi}_1(\lam_1,\lam_2)=0$. Moreover, let us choose 
$S=R=(0,1)$. We see that 
$K^{2R}_{a_2}\cap K^R_{a_1}=\big(K^{2R}_{a_2}\setminus K^R_{a_2}\big)\cap K^R_{a_1}=\{(1,1)\}$ and thus we have  $\xi_2^0(\lam)=\xi^0(\lam)$ for all $\lam\in \Lambda$ except $(1,1)$, for which $\xi^0_2(1,1)=\xi_2(1,1)=1$ and $\xi^0(1,1)=\xi(1,1)=0$.
\end{example}

For each $j=1,...,N$ we define the function $h_j\in C^1_{(1)}(K^{R+S}_{a_j};k)$ as
\begin{equation}\label{eq hj}
h_j(\lam):=-\xi_j(\lam)\cdot \mu_j(\lam)+\xi_j(-S+\lam)\mu_j(\lam)+\mu_j(-R+\lam)\xi_j(\lam)
\end{equation}

and $C^0_j\in C^2_{(1)}(K^{R+S}_{a_j};k)$ as
$$C^0_j(\lam_1,\lam_2):=$$
$$\xi_j^0(\lam_1)\mu_j^0(\lam_2)+\xi_j^0(\lam_2)\mu_j^0(\lam_1)-d\xi_j^0(\lam_1,\lam_2)\mu_j^0(-R+\lam_1+\lam_2)-d\mu_j^0(\lam_1,\lam_2)\xi_j^0(-S+\lam_1+\lam_2).$$

The following proposition is very surprising and it is crucial for later construction of the cup product.

We consider the lattice $M$ as a partially ordered set where positive elements lie in $\Lambda$. Thus if for $\lam\in \Lambda$ we write $\lam\geq R$, it means that $\lam\in \Lambda(R)=\Lambda+R$. 
Equivalently, if for $\lam\in \Lambda$ we write $\lam\not\geq R$, it means that $\lam\in \Lambda\setminus \Lambda(R)$. 
\begin{proposition}\label{prop cup product}
It holds that
\begin{equation}\label{moj th}
d(h_j)=C^0_j\in C^2_{(1)}(K_{a_j}^{R+S};k).
\end{equation}
\end{proposition}
\begin{proof}
From the definition of the differential $d$ in \eqref{def dif d} we have
\begin{align*}
&d(h_j)(\lam_1,\lam_2)=\\
&-\xi_j(\lam_2)\mu_j(\lam_2)+\xi_j(-S+\lam_2)\mu_j(\lam_2)+\mu_j(-R+\lam_2)\xi_j(\lam_2)-\\
&\big( -\xi_j(\lam_1+\lam_2)\mu_j(\lam_1+\lam_2)+\xi_j(-S+\lam_1+\lam_2)\mu_j(\lam_1+\lam_2)+\mu_j(-R+\lam_1+\lam_2)\xi_j(\lam_1+\lam_2) \big)\\
&-\xi_j(\lam_1)\mu_j(\lam_1)+\xi_j(-S+\lam_1)\mu_j(\lam_1)+\mu_j(-R+\lam_1)\xi_j(\lam_1).
\end{align*}

 Recall that by the definition of $C^2_{(1)}(K_{a_j}^{R+S};k)$ we need to verify that $d(h_j)(\lam_1,\lam_2)=C^0_j(\lam_1,\lam_2)$ holds  for those $(\lam_1,\lam_2)\in K_{a_j}^{R+S}\times K_{a_j}^{R+S}$ such that $\lam_1+\lam_2\in K_{a_j}^{R+S}$.

\begin{enumerate}

\item $\lambda_1\not\geq R,S$ and $\lam_2\not\geq R,S$:

\noindent
Note that in this case $\lam_1,\lam_2\in \Lambda\setminus \Lambda(R),\Lambda\setminus \Lambda(S)$. Thus 
 by definition $\xi^0_j(\lam_k)=\xi_j(\lam_k)$ and $\mu^0_j(\lam_k)=\mu_j(\lam_k)$ hold for $k=1,2$. We consider now the following subcases.

\begin{itemize}

\item $\lam_1+\lam_2\geq R,S$:

In this subcase we have $dh_j(\lam_1,\lam_2)=\xi_j(\lam_1)\mu_j(\lam_2)+\xi_j(\lam_2)\mu_j(\lam_1)-\xi_j(-S+\lam_1+\lam_2)\big(\mu_j(\lam_1)+\mu_j(\lam_2)\big )-\mu_j(-R+\lam_1+\lam_2)\big(\xi_j(\lam_1)+\xi_j(\lam_2)\big)$. 

Moreover, $d\xi^0_j(\lam_1,\lam_2)=\xi_j(\lam_1)+\xi_j(\lam_2)$ since $\xi^0_j(\lam_1+\lam_2)=0$. Similarly, $d\mu^0_j(\lam_1,\lam_2)=\mu_j(\lam_1)+\mu_j(\lam_2)$ and thus the equality $dh_j=C^0_j$ follows.

\item $\lam_1+\lam_2\geq R$, $\lam_1+\lam_2\not \geq S$:

$dh_j(\lam_1,\lam_2)=\xi_j(\lam_1)\mu_j(\lam_2)+\xi_j(\lam_2)\mu_j(\lam_1)-\mu_j(-R+\lam_1+\lam_2)\big(\xi_j(\lam_1)+\xi_j(\lam_2)\big).$

Moreover, $d\xi^0_j(\lam_1,\lam_2)=\xi_j(\lam_1)+\xi_j(\lam_2)$ and $d\mu^0_j(\lam_1,\lam_2)=0$ from which the equality $dh_j=C^0_j$ follows.
\item $\lam_1+\lam_2\not\geq R$, $\lam_1+\lam_2\geq S$:

$dh_j(\lam_1,\lam_2)=\xi_j(\lam_1)\mu_j(\lam_2)+\xi_j(\lam_2)\mu_j(\lam_1)-\xi_j(-S+\lam_1+\lam_2)\big(\mu_j(\lam_1)+\mu_j(\lam_2)\big)$. 

\item $\lam_1+\lam_2\not\geq R,S:$

$dh_j(\lam_1,\lam_2)=\xi_j(\lam_1)\mu_j(\lam_2)+\xi_j(\lam_2)\mu_j(\lam_1).$

\end{itemize}
In the last two cases we conclude that $dh_j=C^0_j$ holds in a similar way.

\item $\lambda_1\not\geq R,S$ and $\lam_2\geq R,S:$

We have $\xi_j(\lam_1)=\xi^0_j(\lam_1)$ and $\mu_j(\lam_1)=\mu^0_j(\lam_1)$. Note that these equalities does not necessarily hold for $\lam_2$.
We also know that $\lam_1+\lam_2\geq R,S$ and thus we can easily check that
$dh_j(\lam_1,\lam_2)=
-\mu_j(\lam_1)\big(\xi_j(\lam_1) +\xi_j(-S+\lam_2) \big)-\xi_j(\lam_1)\big( \mu_j(-R+\lam_2)+\mu_j(\lam_1)\big).
$
On the other hand we have 
$C^0_j(\lam_1,\lam_2)=-\xi_j^0(\lam_1)\big( \mu_j^0(-R+\lam_2)+\mu_j^0(\lam_1) \big)-\mu_j^0(\lam_1)\big( \xi_j^0(-S+\lam_2)+\xi_j^0(\lam_1) \big)$. Since $\lam_2\in K_{a_j}^{R+S}$ we have $-R+\lam_2\not\geq S$ and $-S+\lam_2\not\geq R$ and thus $\mu^0_j(-R+\lam_2)=\mu_j(-R+\lam_2)$ and $\xi^0_j(-S+\lam_2)=\xi_j(-S+\lam_2)$. 
It follows that the equality $dh_j=C^0_j$ is also satisfied in this case.

\item 
Similarly as above we can check that the equality \eqref{moj th} is satisfied also in the remaining cases.
\end{enumerate}
\end{proof}

\begin{remark}\label{rem th}
By definition $\xi^0_j(\lam)$ and $\xi^0(\lam)$ can be for $\lam \in K^{R+S}_{a_j}$ different only for $\lam \in \big(K^{R+S}_{a_j}\setminus K^R_{a_j}\big)\cap(\cup_{k;k\ne j} K^R_{a_k})$. The latter space is "not large" (see Example \ref{ex dod3} where it is just a point for $j=2$; later we will also see examples when it is empty).
Similarly, $\mu^0_j(\lam)$ and $\mu^0(\lam)$ can be on $K^{R+S}_{a_j}$ different only for $\lam \in \big(K^{R+S}_{a_j}\setminus K^S_{a_j}\big)\cap(\cup_{k;k\ne j} K^S_{a_k})$. 
Using the notation $C_j$ from Construction \ref{cons 1} we see by the definition that $C_j(\lam_1,\lam_2)=C_j^0(\lam_1,\lam_2)$ if $\xi^0_j(\lam)=\xi^0(\lam)$ and 
$\mu^0_j(\lam)=\mu^0(\lam)$ for $\lam\in K^{R+S}_{a_j}$. 
In Proposition \ref{prop cup product} we verified that $dh_j=C^0_j$ holds, which gives us that $C_j-dh_j$ has many zeros (since $\big(K^{R+S}_{a_j}\setminus K^R_{a_j}\big)\cap(\cup_{k;k\ne j} K^R_{a_k})$ and $\big(K^{R+S}_{a_j}\setminus K^{S}_{a_j}\big)\cap(\cup_{k;k\ne j} K^S_{a_k})$ are "not large"). 
Thus it is easier to find $F_j\in C^1_{(1)}(K^{R+S}_{a_j};k)$ such that $dF_j=C_j-dh_j$ (note that such $F_j$ exists since $H^2_{(1)}(K^{R+S}_{a_j};k)=0$ by Proposition \ref{smooth}). We then define $G_j:=F_j+h_j$ and proceed with Construction \ref{cons 1} to obtain a nice description of the cup product.
\end{remark}
In the next section we will explicitly find the functions $F_j$ (and thus also the functions $G_j$ from Construction \ref{cons 1}) in the special case of Gorenstein toric varieties.

\section{The cup product of affine Gorenstein toric varieties}\label{sec def gor toric var}

Recall that
toric Gorenstein varieties
are obtained by putting a lattice polytope $P\subset \AA$ into the affine hyperplane $\AA\times \{1\}\subset \AA\times \RR=:N_{\RR}$ and defining $\sigma:=\text{Cone}(P)$, the cone over $P$. Then the canonical degree $R^*\in M$ equals $(\underline{0},1)$.

\begin{definition} 
We define the vector space $V\subset \RR^N$ by 
\begin{equation}\label{eq gen v}
V:=V(P):=\{(t_1,...,t_N)~|~\sum_jt_j\epsilon_jd_j=0~|~\text{for every }2\text{-face }\epsilon\leq P\},
\end{equation}
where $\un{\epsilon}=(\epsilon_1,...,\epsilon_N)\in \{0,\pm 1\}^N$ is the sign vector of $\epsilon$ (see \cite[Definition 2.1]{alt3}). \end{definition}

\begin{proposition}\label{prop descr t1}
For Gorenstein toric varieties it holds that $T^1(-R^*)\cong V/k\cdot \underline{1}$. 
\end{proposition}
\begin{proof}
See \cite{alt}.
\end{proof}

For simplicity we will assume that $X_{\sigma}=\spec(A)$ is a three-dimensional Gorenstein singularity given by a cone $\sigma=\lan a_1,...,a_N \ran$, where $a_1,...,a_N$ are arranged in a cycle ($n$-dimensional case for $n>3$ can be then treated with collecting information about all $2$-faces of $P$).

We define $a_{N+1}:=a_1$. 
Let us denote $d_j:=a_{j+1}-a_j$ and let 
\begin{equation}\label{eq defv}
V=\{\underline{t}=(t_1,...,t_N)\in k^N~|~\sum_{j=1}^Nt_jd_j=0\},
\end{equation}
which is the special case of \eqref{eq gen v} in the $3$-dimensional case.
 With $\ell(d_j)$ we will denote the lattice length of $d_j$.

\begin{remark}
Note that if $X_{\sigma}$ is isolated, we have $T^1(-R^*)=T^1$. In general $T^1$ is non-zero also in other degrees (see \cite[Theorem 4.4]{alt}).
\end{remark}

In the following we recall some results from \cite{klaus}, which even simplify the sequence appearing in Theorem \ref{th 24}.

The complex $H^1_{(1)}(K^R_\bullet;k)$ in Theorem \ref{th 24} (for $i=1$) has $(\spa_kK^R_\kbb)^*$
as a subcomplex. Here for $n\in \NN_0$ we define $(\spa_kK^R_n)^*:=\bigoplus_{\tau \leq \sigma,\dim\tau=n}(\spa_kK^R_\tau)^*$ and $(\spa_kK^R_\tau)^*$ denotes the space of linear functions on $\spa_kK^R_\tau$.
\begin{theorem}\label{th end}
 Assume that $X_\sigma=\spec(A)$ is Gorenstein. Then 
$$T^{k,-R}(A)=H^k\big((\spa_kK^R_\kbb)^* \big)$$
for $k=0,1,2$.
\end{theorem}
\begin{proof}
See \cite[Proposition 5.4]{klaus}.
\end{proof}

Note that for every $j=1,...,N$ it holds that 
$$
\sp_kK_{a_j}^R=
\left\{
\begin{array}{ll}
0&\text{ if }\lan a_j,R \ran \leq 0 \\
(a_j)^{\perp} &\text{ if }\lan a_j,R \ran=1\\
M\otimes_{\ZZ}k& \text{ if }\lan a_j,R \ran \geq 2.
\end{array}
\right.
$$

In the following we compute $T^2(-R_m)$ for $R_m=mR^*$ with $m\geq 2$ using Theorem \ref{th end}. For all $j=1,...,N$ it holds that $\spa_kK^{R_m}_{a_j}=M\otimes_Zk$ and $$\spa_k(K^{R_m}_{a_j}\cap K^{R_m}_{a_{j+1}})=(\delta_jd_j)^\perp,$$
where 
$$\delta_j:=\left\{
\begin{array}{ll}
0&\text{ if } \ell(d_j)<m\\
1&\text{ if } \ell(d_j)\geq m.
\end{array}
\right.
$$

Thus the complex $(\spa_kK^{R_m}_\kbb)^*$ for $R_m=mR^*$ with $m\geq 2$ becomes  
$$0\to N_{k}\xrightarrow{\psi} N^N_{k} \xrightarrow{\delta}\oplus_j(N_k/\delta_jd_j)\xrightarrow{\eta} (\sp_k(\cap_jK^{R_m}_{a_j}))^*,$$

where $\psi(x)=(x,...,x)$, $\delta(b_1,...,b_N)=(b_1-b_2,b_2-b_3,...,b_N-b_1)$,
$\eta(q_1,...,q_N)=\sum_{j=1}^Nq_j$.

\begin{proposition}\label{prop descr t2}
It holds that 
$T^2(-R_m)\cong\ker{\eta}/\text{im}~\delta$. 
Moreover, if $m=2$ and $X_{\sigma}$ is isolated, then $\ker{\eta}/\text{im}~\delta\cong (M_{k}/R^*)^*$ holds.
\end{proposition}
\begin{proof}
The first statement follows from the above calculation and Theorem \ref{th end}. The second statement follows from 
the exactness of the complex
$$N^N_{k}\xrightarrow{\delta}N^N_{k}\xrightarrow{\eta}N_{k}.$$
\end{proof}

\subsection{The cup product $T^1(-R^*)\times T^1(-R^*)\to T^2(-2R^*)$}

In the case of isolated three-dimensional toric Gorenstein singularities Altmann \cite{alt2} obtained the following cup product
\begin{equation}\label{eq alt cup}
V/(k\cdot \un{1})\times V/(k\cdot \un{1})\mapsto (M_{k}/R^*)^*
\end{equation}
$$(\underline{t},\underline{s})\mapsto \sum_{j=1}^Ns_jt_jd_j.$$

Recall that it holds $T^1(-R^*)\cong H^2_{(1)}(-R^*)$ and $T^2(-2R^*)\cong H^3_{(1)}(-2R^*)$.
We will now generalize Altmann's cup product formula to the case of not necessarily isolated toric Gorenstein singularities. Note that Altmann was using different methods (Laudal's  cup product) in his proof.

We will first recall the isomorphism map $V/(k\cdot \underline{1})\xrightarrow{f} H^1\big((\spa_kK^{R^*}_\kbb)^* \big)$ from \cite[Section 2.7]{alt}. Note that both spaces are isomorphic to $T^{1,-R^*}(k[\Lambda])$ by Theorem \ref{th end} and Proposition \ref{prop descr t1}. 
By the definition it holds that $(\spa_kK^{R^*}_1)^*=\oplus_j(a^\perp_j)^*$.
We will define $u_j\in (a^\perp_j)^*$ such that $f(\underline{t})=\overline{(u_1,...,u_N)}$ . 
There exist $b_j\in R^{\perp}$ for $j=1,...,N$ such that $\forall j$ it holds that 
\begin{equation}\label{eq proj}
b_{j+1}-b_{j}=t_j(a_{j+1}-a_j).
\end{equation}
Since $\sum_{j=1}^Nt_jd_j=0$ we have a one-parameter solution of this system of equations, namely $b_2=b_1+t_1d_1$, $b_3=b_1+t_1d_1+t_2d_2$,..., $b_N=b_1+\sum_{i=1}^{N-1}t_id_i$. Our function $u_j\in (a^{\perp}_j)^*$ is defined by $u_j(x)=\lan b_j,x\ran$. Note that for different choices of 
$b_1$ we still obtain the same element in $H^1\big((\spa_kK^{R^*}_\kbb)^* \big)$ and thus $f$ is well defined.
Note also that we indeed have $u_j-u_{j+1}=0$ on $a^{\perp}_j\cap a^{\perp}_{j+1}$.

\begin{theorem}\label{prop cup iso}
The cup product $T^1(-R^*)\times T^1(-R^*)\to T^2(-2R^*)$ equals the bilinear map 
$$V/(k\cdot \un{1})\times V/(k\cdot \un{1})\mapsto \ker{\eta}/\im{\delta}$$
$$(\underline{t},\underline{s})\mapsto (s_1t_1d_1,...,s_Nt_Nd_N).$$
\end{theorem}
\begin{proof}
We write for short $R=R^*$. Let 
$\xi=(0,t_1d_1,...,\sum_{j=1}^{n-1}t_jd_j) \in \oplus_jH^1_{(1)}(K^R_{a_j};k)$ and $\mu=(0,s_1d_1,...,\sum_{j=1}^{n-1}s_jd_j)\in \oplus_jH^1_{(1)}(K^R_{a_j};k)$. By the above description of the isomorphism map $f$ (using $b_1=0$) and by Construction \ref{cons 1} it is enough to prove $$\delta G=(t_1s_1d_1,...,t_Ns_Nd_N)\in \bigoplus_{\tau\leq \sigma, \dim\tau=2}H^1_{(1)}(K^{R+S}_\tau;k).$$
Here $G$ is constructed with $\xi$ and $\mu$ as in Construction \ref{cons 1}.

By the description of $\xi$ and $\mu$ we define 
$\xi_j,\mu_j\in C^1_{(1)}(M;k)$ as $\xi_1=\mu_1=0$ and 
$$\xi_j(\lam):=
\left\{
\begin{array}{ll}
\lan \sum_{k=1}^{j-1}t_kd_k,\lam \ran&\text{ if }\lam\in K_{a_j}^{2R}\\
0&\text{ otherwise }
\end{array}
\right.
,~~~~
\mu_j(\lam):=
\left\{
\begin{array}{ll}
\lan \sum_{k=1}^{j-1}s_kd_k,\lam \ran&\text{ if }\lam\in K_{a_j}^{2R}\\
0&\text{ otherwise }
\end{array}
\right.
$$
for $2\leq j\leq N$.
Recall from the equation \eqref{eq hj} that the functions $h_j\in C^1_{(1)}(K^{2R}_{a_j};k)$ are
$$h_j(\lam)=-\xi_j(\lam)\mu_j(\lam)+\xi_j(-R+\lam)\mu_j(\lam)+\mu_j(-R+\lam)\xi_j(\lam).$$

Recall that $dh_j=C^0_j$ by Proposition \ref{prop cup product}.
From Remark \ref{rem th} we know that in order to determine when $(dh_j-C_j)(\lam_1,\lam_2)\in C^2_{(1)}(K^{2R}_{a_j};k)$ is zero, we need to consider the space  $P^j:=(K^{2R}_{a_j}\setminus K^R_{a_j})\cap(\cup_{k;k\ne j} K^R_{a_k})$. Denoting
$$P^j_1:=(K^{2R}_{a_j}\setminus K^R_{a_j})\cap K^{R}_{a_{j+1}},$$
$$P^j_2:=(K^{2R}_{a_j}\setminus K^R_{a_j})\cap K^{R}_{a_{j-1}},$$
we see that 
$P^j=P^j_1\cup P^{j}_2$ for each $j$.

 If $\ell(d_j)>1$ (or equivalently $\lan a_j,a_{j+1} \ran$ is not smooth), then $P^j_1=P^j_2=\emptyset$ and thus  
$dh_j=C_j\in C^2_{(1)}(K_{a_j}^{2R};k)$ for all $j=1,...,N$ by Proposition \ref{prop cup product} and Remark \ref{rem th}.

If $\ell(d_j)=1$ (or equivalently if $\lan a_j,a_{j+1} \ran$ is smooth), then $P^j_1\subset \Lambda$ and $P^j_2\subset \Lambda$ are infinite sets contained in the lines parallel to $a_j^{\perp}\cap a_{j+1}^{\perp}$ and $a_{j-1}^{\perp}\cap a_{j}^{\perp}$, respectively.
If $\lam \in P_1^j\cup P_2^j$, then $\lan \lam,a_j\ran=1$.
We want to find the functions $F_j\in C^1_{(1)}(K_{a_j}^{2R};k)$ for which $dh_j+dF_j=C_j\in C^2_{(1)}(K_{a_j}^{2R};k)$ holds. 
Let 
$$
F_j(c):=\left\{
\begin{array}{ll}
-\xi(c)s_jd_j(c)-\mu(c)t_jd_j(c)=\xi(c)s_j+\mu(c)t_j&\text{ if }c\in P^j_1\\
\xi(c)s_{j-1}d_{j-1}(c)+\mu(c)t_{j-1}d_{j-1}(c)=-\xi(c)s_{j-1}-\mu(c)t_{j-1}&\text{ if }c\in P^j_2\\
0&\text{ otherwise}.
\end{array}
\right.
$$
Using that $dh_j=C^0_j$ (see Proposition \ref{prop cup product}) it is straightforward to verify that $dh_j+dF_j=C_j$ indeed holds. Let us just verify it for 
$(\lam^j,\lam_1^j)\in (\Lambda\cap a_j^\perp\cap a_{j+1}^\perp,P^j_1)$. In this case we have $$dF_j(\lam^j,\lam_1^j)=-\xi(\lam^j)s_j-\mu(\lam^j)t_j,$$ $$(C_j-dh_j)(\lam^1,\lam_1^j)=\xi^0(\lam^j)\mu^0(\lam_1^j)+\mu^0(\lam^j)\xi^0(\lam_1^j)-\xi_j^0(\lam^j)\mu_j^0(\lam_1^j)-\mu_j^0(\lam^j)\xi_j^0(\lam_1^j).$$ Since $\xi^0(\lam^j)=\xi_j^0(\lam^j)$, $\mu^0(\lam^j)=\mu_j^0(\lam^j)$ and $(\xi^0-\xi^0_j)(\lam_1^j)=t_jd_j(\lam_1^j)=-t_j, (\mu^0-\mu^0_j)(\lam_1^j)=s_jd_j(\lam_1^j)=-s_j$, we conclude the claim.

To conclude the proof of Theorem \ref{prop cup iso}, we need to show that 
\begin{equation}\label{eq zadnazad}
\delta G=\delta(F_1,...,F_N)+\delta(h_1,...,h_N)=(t_1s_1d_1,...,t_Ns_Nd_N)\in H^1_{(1)}(K^{2R}_{j}\cap K^{2R}_{j+1};k).
\end{equation}

Since $K^{2R}_{j}\cap K^{2R}_{j+1}=P_1^j\cup P_2^{j+1}\cup \big(kR+\Lambda\cap a_j^{\perp}\cap a_{j+1}^{\perp}\big)$ for $k\in \{0,1\}$, we need to distinguish the following cases.
\begin{enumerate}
\item $c\in P_1^j$: it holds that $\lan a_j,c\ran=1$, $\lan a_{j+1},c\ran=0$ and thus we have $F_{j+1}(c)=0$, $F_j(c)=\xi(c)s_j+\mu(c)t_j$. In this case we have $\xi_{j+1}(c)=\xi(c)$. Using $\xi_{j+1}(c)=\xi_j(c)+t_jd_j(c)$,
$
(h_{j}-h_{j+1})(c)=-\xi_j(c)\mu_j(c)+\xi_{j+1}(c)\mu_{j+1}(c)$ and the fact that 
$d_j(c)=-1$ we obtain that
$$(F_{j}-F_{j+1})(c)+(h_{j}-h_{j+1})(c)=-s_jt_j.$$  
\item $c\in P_2^{j+1}$: it holds that $\lan a_{j},c\ran=0$,  $\lan a_{j+1},c\ran=1$ and thus we have $F_j(c)=0$, $F_{j+1}(c)=-\xi(c)s_j-\mu(c)t_j$. In this case we have $\xi_j(c)=\xi(c)$ and similarly as in the case (1) we obtain that
$$(F_{j}-F_{j+1})(c)+(h_{j}-h_{j+1})(c)=t_js_j.$$
\item $c\in kR+\Lambda\cap a_j^{\perp}\cap a_{j+1}^{\perp}$: it holds that $(F_{j}-F_{j+1})(c)+(h_{j}-h_{j+1})(c)=0$.

\end{enumerate}
Since $d_j(c)=-1$ for $c\in P_1^j$ and $d_j(c)=1$ for $c\in P_2^{j+1}$ we see that the equation \eqref{eq zadnazad} indeed holds.  Note that  if $\ell(d_j)>1$, then we already mentioned that $P^j_1=P^j_2=\emptyset$ for all $j$, which implies $F_j=F_{j+1}=0$ and as in the case (3) above we have $\delta\underline{h}=0$ which agrees with our cup product formula since $t_js_jd_j=0$ on $N_k/d_j$. Thus we conclude the proof.
\end{proof}
 \begin{corollary}\label{cor altmann}
If $X_{\sigma}$ is an isolated Gorenstein singularity, then Theorem \ref{prop cup iso} gives us Altmann's cup product \eqref{eq alt cup}. 
 \end{corollary}
\begin{proof}
Follows from Proposition \ref{prop descr t2} since the isomorphism $\ker \mu/\im \delta\cong (M_k/R^*)$ is explicitly given by the summation of components in $\ker\mu\in N^N_k$.
\end{proof}

We will denote the cup product from Theorem \ref{prop cup iso} by $\un{t}\cup \un{s}\in \ker \mu/\im \delta$.

\begin{corollary}\label{cor quad}
\begin{enumerate}
\item  If all edges of $P$ have lattice length 1 (i.e. $X_\sigma$ is an isolated singularity), then we have $\un{t}\cup \un{s}=0$ if and only if $d_1t_1s_1+\cdots+ d_Nt_Ns_N=0$.
\item $P$ has one edge $d_j$ with $\ell(d_j)>1$ and for all other edges $d_k$ holds either $\ell(d_k)=1$ for all $k\ne j$ or there exists $d_i$ parallel to $d_j$ with $\ell(d_i)>1$. In this case we have $\un{t}\cup \un{s}=0$ if and only if $d_1t_1s_1+\cdots+ d_Nt_Ns_N=0$ on $a_j^\perp\cap a^\perp_{j+1}=a_i^\perp\cap a^\perp_{i+1}$.  

\item $P$ has at least two non parallel edges that have lattice length $\geq 2$. In this case $\un{t}\cup \un{s}=0$ always holds.
\end{enumerate}
\end{corollary}
\begin{proof}
From the definition of the map $\delta$ we see that $\un{t}\cup \un{s}=0$ if and only if for all $j$ such that $\ell(d_j)>1$ there exist functions $f_j(\un{t},\un{s})$ such that 
$$\big(\sum_{j; \ell(d_j)=1}t_js_jd_j\big)+\big( \sum_{j; \ell(d_j)>1}f_j(\un{t},\un{s})d_j\big)=0.$$
From this the proof easily follows.
\end{proof}

\begin{remark}
By standard deformation theory arguments (see e.g. \cite[pp. 64]{ste}) we know that the quadratic equations corresponding to $\un{t}\cup \un{t}=0$ give us the quadratic equations of the versal base space in degree $-R^*$.
\end{remark}

\subsection{The cup product between nonnegative degrees}

Let $X_{\sigma}$ be a non isolated three-dimensional toric Gorenstein singularity. 
In this section we compute the cup product $T^1(-R)\times T^1(-S)\to T^2(-R-S)$ for $R,S\not \geq 0$, i.e. for $R,S\not\in \Lambda$. If $R$ and $S$ have the last component equal to $0$, then the computations in this section have implications in deformation theory of projective toric varieties.

Let $s_1,...,s_N$ be the fundamental generators of the dual cone $\sigma^{\vee}$, labelled so that $\sigma \cap (s_j)^{\perp}$ equals the face spanned by $a_j,a_{j+1}\in \sigma$.

Let $R^{p,q}_j:=qR^*-ps_j$ with $2\leq q \leq \ell(d_j)$ and $p\in \ZZ$ sufficiently large such that $R^{p,q}_j\not \in \text{int}(\sigma^{\vee})$. In this case we already know that $T^1(-R_j^{p,q})$ is one dimensional by \cite[Theorem 4.4]{alt}. 

\begin{lemma}\label{lem more 2}
If $\#\{a_j~|~\lan a_j,R\ran> 0\}\leq 2$ it holds that $T^2(-R)=0$. 
\end{lemma}
\begin{proof}
If $\#\{a_j~|~\lan a_j,R\ran> 0\}\leq 1$, the statement is trivial.
Without loss of generality we assume $\lan a_i,R\ran> 0$ for $i=1,2$ and $\lan a_j,R\ran\leq 0$ for other $j$. Now the statement follows from the fact that $T^2=0$ for a Gorenstein surface $\lan a_1,a_2 \ran\subset N_{\RR}\cong \RR^2$.
\end{proof}

\begin{proposition}
Let $R_1:=R^{p_1,q_1}_j$ and $R_2:=R^{p_2,q_2}_k$, where $j$ and $k$ are chosen such that either $j=k$ or it does not exist a $2$-face $F$ of $\sigma^\vee$ with $s_j,s_k\in F$.
The cup product $T^1(-R_1)\times T^1(-R_2)\to T^2(-R_1-R_2)$ is the zero map.
\end{proposition}
\begin{proof}
We will use computations obtained in Section 3 (see Construction \ref{cons 1} and Remark \ref{rem th}).
Let $\xi\in H^1_{(1)}(\Lambda\setminus \Lambda(R_1);k)$ and $\mu\in H^1_{(1)}(\Lambda\setminus \Lambda(R_2);k)$ represent basis elements for $T^1(-R_1)$ and $T^1(-R_2)$, respectively. Note that $\lan a_j,R_1\ran$, $\lan a_{j+1},R_1\ran>0$ and $\lan a_i,R_1\ran\leq 0$ for $i\ne j,j+1$. Similarly, $\lan a_k,R_2\ran$, $\lan a_{k+1},R_2\ran>0$ and $\lan a_i,R_2\ran\leq 0$ for $i\ne k,k+1$.
If $j=k$, then the statement follows from Lemma \ref{lem more 2}.
Otherwise, we know by the assumption that it holds
 $\lan a_j,R_1+R_2\ran\leq \lan a_j,R_1\ran$ 
and thus $K^{R_1+R_2}_{a_j}\subset K^{R_1}_{a_j}$ holds. By the assumptions we also have 
$$
K^{R_1+R_2}_{a_{j+1}}\subset K^{R_1}_{a_{j+1}}, ~K^{R_1+R_2}_{a_{k}}\subset K^{R_2}_{a_{k}}, ~K^{R_1+R_2}_{a_{k+1}}\subset K^{R_2}_{a_{k+1}}.
$$
This implies that $\xi^0_j(\lam)=\xi^0(\lam)$ for $\lam\in K^{R_1+R_2}_{a_{j}}$, 
$\xi^0_{j+1}(\lam)=\xi^0(\lam)$ for $\lam\in K^{R_1+R_2}_{a_{j+1}}$,  
$\mu^0_k(\lam)=\mu^0(\lam)$ for $\lam \in K^{R_1+R_2}_{a_k}$ and 
$\mu^0_{k+1}(\lam)=\mu^0(\lam)$ for $\lam \in K^{R_1+R_2}_{a_{k+1}}$.
By Construction \ref{cons 1} and Remark \ref{rem th} then follows that the cup product is equal to $\delta(h_1,...,h_N)\in \oplus_{j=1}^NH^1_{(1)}(K_{j,j+1}^{R_1+R_2};k)$, which is clearly equal to zero (since $h_j=0$ for all $j$).
\end{proof}

The following example shows that we can also compute the cup product between the elements of degrees $R_1:=R^{p_1,q_1}_j$ and $R_2:=R^{p_2,q_2}_{j+1}$. 

\begin{example}
A typical example of a non-isolated, three dimensional toric Gorenstein singularity is the affine cone $X_{\sigma}$ over the weighted projective space $\PP(1,2,3)$. The cone $\sigma$ is given by 
$\sigma=\lan a_1,a_2,a_3\ran$, where 
$$a_1=(-1,-1,1),~~~ a_2=(2,-1,1),~~~a_3=(-1,1,1).$$
We obtain $\sigma^{\vee}=\lan s_1,s_2,s_3 \ran$ with 
$$s_1=(0,1,1),~~~s_2=(-2,-3,1),~~~s_3=(1,0,1).$$
$H^2_{(1)}$ is nonzero in degrees $R^1_{\alpha}:=2R^*-\alpha s_3$, $R^2_{\beta}:=2R^*-\beta s_1$ and $R^3_{\gamma}:=2R^*-\gamma s_1$ with $\alpha\geq 1$, $\beta\geq 1$ and $\gamma \geq 2$.
Let us denote the corresponding basis element of $R^1_{\alpha}$, $R^2_{\beta}$ and $R^3_{\gamma}$ by $z^1_{\alpha}$, $z^2_{\beta}$ and $z^3_{\gamma}$, respectively.

We have 
$$\lan a_1,R^1_{\alpha}\ran=\lan a_3,R^1_{\alpha}\ran=2,~~~ \lan a_2,R^1_{\alpha}\ran=2-3\alpha,$$
$$\lan a_1,R^2_{\beta}\ran=\lan a_2,R^2_{\beta}\ran=2,~~~ \lan a_3,R^2_{\beta}\ran=2-2\beta,$$
$$\lan a_1,R^3_{\gamma}\ran=\lan a_2,R^3_{\gamma}\ran=3,~~~ \lan a_3,R^3_{\gamma}\ran=3-2\gamma.$$
From Lemma \ref{lem more 2} we know that the only possible nonzero cup products can be $z^1_1\cup z^2_1$ and $z^1_1\cup z^3_2$, since in other cases we have $T^2(R^i_j+R^k_l)=0$. 
Using computations in Section 3 (more precisely using Construction \ref{cons 1} and Remark \ref{rem th}) we can easily verify that it indeed holds $z^1_1\cup z^2_1\ne 0$ and $z^1_1\cup  z^3_2\ne 0$. In this case the equations $z^1_1\cdot z^2_1=z^1_1\cdot z^3_2=0$ are already the  generalized (infinite dimensional) versal base space. J. Stevens checked this using the computer algebra system Macaulay, see \cite[Section 5.2]{alt}. 
\end{example}

\section{The Gerstenhaber product $\HH^2(k[\Lambda])\times \HH^2(k[\Lambda])\to \HH^3(k[\Lambda])$}

Recall from the discussion after Lemma \ref{lem ger pr} that the only remaining missing part for understanding the Gerstenhaber product $\HH^2(k[\Lambda])\times \HH^2(k[\Lambda])\to \HH^3(k[\Lambda])$ is the product $H^2_{(1)}(k[\Lambda])\times H^2_{(2)}(k[\Lambda])\to H^3_{(2)}(k[\Lambda])$. 
We will analyse it in this section.

As before let $R,S\in M$.
Every element from $H^2_{(1)}(\Lambda,\Lambda\setminus \Lambda(S))$ can be written as $d\mu^0$ for some $$\mu\in H_{(1)}^1(\Lambda\setminus \Lambda(S);k).$$
By \cite[Proposition 4.7]{fil} we know that $H^2_{(2)}(\Lambda,\Lambda\setminus \Lambda(R))$ is isomorphic to the space of multi-additive, skew-symmetric functions $f:\Lambda\times \Lambda\to \Lambda$, such that $f(\lam_1,\lam_2)=0$ if $\lam_1+\lam_2\in \Lambda\setminus \Lambda(R)$. 
\begin{remark}
Multi-additivity means that $f(a+b,c)=f(a,c)+f(b,c)$ and $f(a,b+c)=f(a,b)+f(a,c)$ hold for all $a,b,c\in \Lambda$.
\end{remark}

From Lemma \ref{lem ger pr} recall the description of the Gerstenhaber product.

\begin{proposition}\label{prop com notcom}
Let $\mu\in H^1_{(1)}(\Lambda\setminus \Lambda(S);k)$ and $\xi\in H^2_{(2)}(\Lambda,\Lambda\setminus\Lambda(R);k)$. Let
$$B(\lam_1,\lam_2):=B_1(\lam_1,\lam_2)-B_2(\lam_1,\lam_2)\in C^2_{(2)}(\Lambda;k),$$ where 
$$B_1(\lam_1,\lam_2):=\xi(-S+\lam_1+\lam_2,\lam_2)\mu^0(\lam_1)+\xi(\lam_1,-S+\lam_1+\lam_2)\mu^0(\lam_2),$$
$$B_2(\lam_1,\lam_2):=\xi(\lam_1,\lam_2)\mu^0(\lam_1+\lam_2-R).$$
Let $\lam_{123}:=\lam_1+\lam_2+\lam_3$.

\begin{enumerate}

\item If $\lam_1+\lam_2\geq S$, $\lam_2+\lam_3\geq S$ we have $dB(\lam_1,\lam_2,\lam_3)=[\xi,d\mu^0](\lam_1,\lam_2,\lam_3)$.

\item If $\lam_1+\lam_2\not \geq S$, $\lam_2+\lam_3\geq S$ we have 
\begin{align*}
&\big(dB-[\xi,d\mu^0]\big)(\lam_1,\lam_2,\lam_3)=\\
&\mu^0(\lam_1)\big(\xi(-S+\lam_{123},\lam_2)+\xi(\lam_2,\lam_3)\big)
+\mu^0(\lam_2)\big(\xi(\lam_1,-S+\lam_{123}) -\xi(\lam_1,\lam_3)\big).
\end{align*}

\item If $\lam_1+\lam_2 \geq S$, $\lam_2+\lam_3\not\geq S$ we have 
\begin{align*}
&\big(dB-[\xi,d\mu^0]\big)(\lam_1,\lam_2,\lam_3)=\\
&\mu^0(\lam_2)\big(\xi(\lam_1,\lam_3)-\xi(-S+\lam_{123},\lam_3)\big)+
\mu^0(\lam_3)\big(\xi(-S+\lam_{123},\lam_2)-\xi(\lam_1,\lam_2)\big).
\end{align*}

\item If $\lam_1+\lam_2 \not \geq S$, $\lam_2+\lam_3\not \geq S$ we have
\begin{align*}
&\big(dB-[\xi,d\mu^0]\big )(\lam_1,\lam_2,\lam_3)= 
\mu^0(\lam_1)\big(\xi(-S+\lam_{123},\lam_2)+\xi(\lam_2,\lam_3) \big)+\\
&+\mu^0(\lam_2)\big(\xi(\lam_1,-S+\lam_{123})-\xi(-S+\lam_{123},\lam_3) \big)+\\
&+\mu^0(\lam_3)\big(\xi(-S+\lam_{123},\lam_2) -\xi(\lam_1,\lam_2) \big).
\end{align*}

\end{enumerate}

\end{proposition}
\begin{proof}

It holds that
\begin{align*}
&[\xi,d\mu^0](\lam_1,\lam_2,\lam_3)=\\
&=\mu^0(\lam_1)\big(\xi(-S+\lam_1+\lam_2,\lam_3)-\xi(\lam_2,\lam_3)\big)+\mu(\lam_2)\big(\xi(-S+\lam_1+\lam_2,\lam_3)-\xi(\lam_1,-S+\lam_2+\lam_3)\big)\\
&+\mu^0(\lam_3)\big(-\xi(\lam_1,-S+\lam_2+\lam_3)+\xi(\lam_1,\lam_2)\big)-\mu^0(\lam_1+\lam_2)\xi(-S+\lam_1+\lam_2,\lam_3)\\
&+\mu^0(\lam_2+\lam_3)\xi(\lam_1,-S+\lam_2+\lam_3)-dB_2(\lam_1,\lam_2,\lam_3),
\end{align*} 
where we used the fact that $\xi$ is multi-additive.
We will conclude the proof case by case.

\begin{enumerate}

\item $\lam_1+\lam_2\geq S$, $\lam_2+\lam_3\geq S$

We have $\mu^0(\lam_1+\lam_2)=\mu^0(\lam_2+\lam_3)=0$. Thus we compute
\begin{align*}
&dB_1(\lam_1,\lam_2,\lam_3)=\\
&=\mu^0(\lam_1)\big(\xi(-S+\lam_1+\lam_2,\lam_3)+\xi(\lam_3,\lam_2)\big)+\\
&+\mu^0(\lam_2)\big(\xi(-S+\lam_2+\lam_3,\lam_3)-\xi(\lam_1,-S+\lam_1+\lam_2)\big)+\\
&+\mu^0(\lam_3)\big(-\xi(\lam_1,-S+\lam_2+\lam_3)-\xi(\lam_2,\lam_1)\big).
\end{align*}
It holds that
\begin{align*}
&\xi(-S+\lam_2+\lam_3,\lam_3)-\xi(\lam_1,-S+\lam_1+\lam_2)=\\
&=\xi(-S+\lam_1+\lam_2+\lam_3,\lam_3)-\xi(\lam_1,\lam_3)\\
&-\xi(\lam_1,-S+\lam_1+\lam_2+\lam_3)+\xi(\lam_1,\lam_3)=\\
&=\xi(-S+\lam_1+\lam_2,\lam_3)-\xi(\lam_1,-S+\lam_2+\lam_3)
\end{align*}
and thus 
we see that in this case $dB(\lam_1,\lam_2,\lam_3)=[\xi,d\mu^0](\lam_1,\lam_2,\lam_3)$ holds.

\item $\lam_1+\lam_2\not \geq S$, $\lam_2+\lam_3\geq S$:

We have $\mu^0(\lam_2+\lam_3)=0$ and $\mu^0(\lam_1+\lam_2)=\mu^0(\lam_1)+\mu^0(\lam_2)$.  It holds that 
\begin{align*}
&dB_1(\lam_1,\lam_2,\lam_3)=\\
&=\mu^0(\lam_1)\xi(-S+\lam_1+\lam_2+\lam_3,\lam_2)+\mu^0(\lam_2)\big( \xi(-S+\lam_2+\lam_3,\lam_3)-\xi(-S+\lam_1+\lam_2+\lam_3,\lam_3) \big)+
\\&+\mu^0(\lam_3)\big( \xi(\lam_2,-S+\lam_2+\lam_3)-\xi(\lam_1+\lam_2,-S+\lam_1+\lam_2+\lam_3) \big),
\\
\\
&[\xi,d\mu](\lam_1,\lam_2,\lam_3)=\\ 
&=\mu^0(\lam_1)(-\xi(\lam_2,\lam_3))+\mu^0(\lam_2)(-\xi(\lam_1,-S+\lam_2+\lam_3))+\\
&+\mu^0(\lam_3)\big(-\xi(\lam_1,-S+\lam_2+\lam_3)+\xi(\lam_1,\lam_2)\big).
\end{align*}
If we compute $\big(dB_1-[\xi,d\mu^0]\big)(\lam_1,\lam_2,\lam_3)$ we see that the term before $\mu^0(\lam_3)$ vanishes because
\begin{align*}
&\xi(\lam_2,-S+\lam_2+\lam_3)-\xi(\lam_1+\lam_2,-S+\lam_1+\lam_2+\lam_3)=\\
&=\xi(\lam_2,-S+\lam_1+\lam_2+\lam_3)-\xi(\lam_2,\lam_1)-\xi(\lam_1+\lam_2,-S+\lam_1+\lam_2+\lam_3)=\\
&=-\xi(\lam_1,-S+\lam_2+\lam_3)+\xi(\lam_1,\lam_2).
\end{align*}
This gives us the desired result.
\end{enumerate}
Similarly we can consider the remaining cases.
\end{proof}

\begin{corollary}\label{corollary non cup product}
 The product
 $[d\mu^0,\xi]\in H^3_{(2)}(-R-S)$  is equal to the cohomological class of the element
 $$(\delta(B),d(B)-[d\mu^0,\xi])\in C^2_{(2)}(K_1^{R+S};k)\oplus C^3_{(2)}(\Lambda;k)$$
 in the total complex of the complex $C^{\kbb}_{(2)}(K^R_{\kbb};k)$ (see the equation \eqref{eq im double}). 
Note that the map $d(B)-[d\mu^0,\xi]$ has many zeros by Proposition \ref{prop com notcom} and thus it is easier to compute it. 
\end{corollary}
In the following we will compute it in a special case of toric surfaces and found out that the Gerstenhaber product is the zero map.

Let $X_{\sigma_n}=\spec(A_n)$ be the Gorenstein toric surface given by $g(x,y,z)=xy-z^{n+1}$. $\Lambda_n:=\sigma_n^{\vee}\cap M$ is generated by $S_1:=(0,1)$, $S_2:=(1,1)$ and $S_3:=(n+1,n)$, with the relation $S_1+S_3=(n+1)S_2$. 
We recall now from \cite[Example 3]{fil} that 
we have 
\begin{equation}\label{eq h32}
\dim_kH^2_{(1)}(-R)=\dim_kH^3_{(2)}(-R)=\left\{
\begin{array}{ll}
1& \text{ if }R=kS_2 \text{ for }2\leq k\leq n+1\\ 
0 & \text{ otherwise, }
\end{array}
\right.
 \end{equation}
 $$H^3_{(3)}(A_n)=H^3_{(1)}(A_n)=0.$$

 Let $\{\overline{d\mu^0_k}\in H^{2,-kS_2}_{(1)}(A_n)~|~2\leq k\leq n+1\}$ be a basis of $T^1_{(1)}(A_n)\cong H^2_{(1)}(A_n)$, such that  $\mu_k\in C^1_{(1)}(\Lambda\setminus \Lambda(kS_2);k)$ is defined by
$$\mu_k(\lam)=\left\{
\begin{array}{ll}
a& \text{ if }\lam=aS_3, \text{ for }a\in \NN\\ 
0 & \text{ otherwise }
\end{array}
\right.
$$
for all $2\leq k\leq n+1$.

\begin{proposition}\label{pr gor sur}
For all toric Gorenstein surfaces $\spec(A_n)$ it holds that the Gerstenhaber product  $\HH^2(A_n)\times \HH^2(A_n)\to \HH^3(A_n)$ is equal to the zero map.
\end{proposition}
\begin{proof}
We choose some $\xi\in H^2_{(2)}(\Lambda,\Lambda\setminus \Lambda(R))$ for arbitrary $R\in M$.
Using Proposition \ref{prop com notcom} we will show that $dB=[d\mu^0_k,\xi]$ holds 
for all $k$.
Indeed, considering the cases (2), (3) and (4) of Proposition \ref{prop com notcom} we see that in the case (2) we need to prove that for $\lam_1+\lam_2\not\geq S$, $\lam_2+\lam_3\geq S$ it holds that
\begin{equation}\label{eq lep cup}
\mu^0(\lam_1)\big(\xi(-S+\lam_{123},\lam_2)+\xi(\lam_2,\lam_3)\big)+\mu^0(\lam_2)\big(\xi(\lam_1,-S+\lam_{123}) -\xi(\lam_1,\lam_3)\big)=0.
\end{equation}
Let $a_1,a_2$ be arbitrary natural numbers.
We choose $\lam_1=a_1S_3$, $\lam_2=a_2S_3$ and $\lam_3$ arbitrary such that $\lam_2+\lam_3\geq S$ (note that in all other cases the equation \eqref{eq lep cup} holds trivially). Now that the equation \eqref{eq lep cup} holds  follows from multi-additivity and skew-symmetry of $\xi$:
$$\mu^0(\lam_1)\big(\xi(-S+\lam_{123},\lam_2)+\xi(\lam_2,\lam_3)\big)+\mu^0(\lam_2)\big(\xi(\lam_1,-S+\lam_{123}) -\xi(\lam_1,\lam_3)\big)=$$
$$a_1(\xi(-S+\lam_2+\lam_3,\lam_2)+\xi(\lam_1,\lam_2)+\xi(\lam_2,\lam_3))+a_2(\xi(\lam_1,-S+\lam_2+\lam_3)-\xi(\lam_1,\lam_3))=$$
$$
a_1(a_2\xi(-S+\lam_2+\lam_3,S_3)+a_2\xi(S_3,\lam_3))+a_2(a_1\xi(S_3,-S+\lam_2+\lam_3)-a_1\xi(S_3,\lam_3))=0,
$$
where in the second equality we used that $\xi(\lam_1,\lam_2)=a_1a_2\xi(S_3,S_3)=0$ and in the last equality we used that $\xi(-S+\lam_2+\lam_3,S_3)=-\xi(S_3,-S+\lam_2+\lam_3)$.  

Similarly we can check the cases (3) and (4) and we obtain the claim that $dB=[d\mu^0_k,\xi]$ for all $k$.

From the description of $H^3_{(2)}(A_n)$ from the equation \eqref{eq h32}, we see that we only need to consider the cases, where $R$ is chosen such that $R+S\in \{kS_2~|~ \text{ for } 1\leq k\leq n+1\}$, since in all other cases $\HH^{3,-R-S}(A_n)=H^{3,-R-S}_{(2)}(A_n)=0$ so the Gerstenahber bracket is automatically zero.    
If $R+S\in \{kS_2~|~ \text{ for } 1\leq k\leq n+1\}$, then we easily see that $B_1(\lam_1,\lam_2)=B_2(\lam_1,\lam_2)=B(\lam_1,\lam_2)=0$ for all $(\lam_1,\lam_2)\in \Lambda\times \Lambda$ such that $\lam_1+\lam_2\not\geq R+S$. By Corollary \ref{corollary non cup product} we conclude the proof.
\end{proof}
\section*{Acknowledgements}
I would like to thank to my PhD advisor Klaus Altmann,
for his constant support and for providing clear answers to my many questions.
I am also grateful to Arne B. Sletsj\o e for useful discussions.


\begin{thebibliography}{44}



\bibitem{alt2}
{K. Altmann: \emph{Infinitesimal deformations and obstructions for toric singularities.} J. Pure Appl. Alg. \textbf{119} (1997), 211--235}.

\bibitem{alt3}
{K. Altmann: \emph{The versal deformation of an isolated, toric Gorenstein singularity.} Invent. Math. \textbf{128} (1997), 443--479}.

\bibitem{alt}
{K. Altmann:  \emph{One parameter families containing three-dimensional toric Gorenstein singularities}, Explicit birational geometry of 3-folds, London Math. Soc. Lecture Note Ser., vol. \textbf{281}, Cambridge Univ. Press, Cambridge (2000), 21--50}.


\bibitem{klaus}
{K. Altmann, A. B. Sletsj\o e: \emph{Andr\'e-Quillen cohomology of monoid algebras} J.  Alg. \textbf{210} (1998), 1899--1911}.




\bibitem{chr}
{J.A. Christophersen: \emph{On the components and discriminant of the versal base space of cyclic quotient singularities}. In Singularity theory and its applications, Part I (Coventry 1988/1989), volume 1462 of Lecture Notes in Math., Springer, Berlin (1991), 81--92}.



\bibitem{cox}
{D. A. Cox, J. B. Little, H.K. Schenk: \emph{Toric varieties}, Graduate Studies in Mathematics \textbf{124}, AMS (2011)}.


\bibitem{dol-tam-tsy}
{V. Dolgushev, D. Tamarkin, B. Tsygan: \emph{The homotopy Gerstenhaber algebra of Hochschild cochains of a regular algebra is formal}, J. Noncommut. Geom. \textbf{1} (2007), iss. 1, 1--25.}

\bibitem{fil}
{M. Filip: \emph{Hochschild cohomology and deformation quantization of affine toric varieties}, J. Alg. \textbf{508} (2018), 188--214}.







\bibitem{kol-she}
{J. Koll\'ar, N.I. Shepherd-Barron: \emph{Threefolds and deformations of surface singularities}, Invent. Math. \textbf{91}, (1988), 299-338}.

\bibitem{kon2}
{M. Kontsevich: \emph{Deformation quantization of algebraic varieties}, Lett. Math. Phys. \textbf{56} (2001), iss. 3, 271--294.}



\bibitem{lod}
{J.-L. Loday: \emph{Cyclic homology}, Grundlehren der mathematischen Wissenschaften \textbf{301}, Springer-Verlag, (1992)}.




\bibitem{pal2}
{V.P. Palamodov: \emph{Infinitesimal deformation quantization of complex analytic spaces}, Lett. Math. Phys. \textbf{79} (2007), iss. 2, 131--142.}

\bibitem{sle}
{A. B. Sletj\o e: \emph{Cohomology of monoid algebras}, J. Alg. \textbf{161} (1993), 102--1911}.

\bibitem{ste}
{J. Stevens: \emph{Deformations of singularities}, Lecture notes in Mathematics \textbf{1811}, Springer, Berlin (2003).}

\bibitem{ste2}
{J. Stevens: \emph{On the versal deformation of cyclic quotient singularities}, Singularity theory and its applications, Part I (Coventry 1988/1989), volume 1462 of Lecture Notes in Math., Springer, Berlin (1991), 302-319}.


\end{thebibliography}
\end{document}